\documentclass[11pt]{article}
\usepackage[margin=1in]{geometry} 
\usepackage[utf8]{inputenc}
\usepackage{amssymb,amsfonts,amsmath,bbm,mathrsfs,stmaryrd}
\usepackage{mathabx}
\usepackage{xcolor}
\usepackage{url}
\usepackage{ogonek} 

\usepackage[all,cmtip]{xy}

\xyoption{all} 

\usepackage[colorlinks,
             linkcolor=black!75!red,
             citecolor=blue,
             pdftitle={hs},
             pdfauthor={Alexandru Chirvasitu},
             pdfproducer={pdfLaTeX},
             pdfpagemode=None,
             bookmarksopen=true
             bookmarksnumbered=true]{hyperref}

\usepackage{tikz,tikz-cd}
\usetikzlibrary{arrows,calc,decorations.pathreplacing,decorations.markings,intersections,shapes.geometric,through,fit,shapes.symbols,positioning,decorations.pathmorphing}
\usepackage{easyReview}

\usepackage[amsmath,thmmarks,hyperref]{ntheorem}
\usepackage{cleveref}

\crefname{section}{Section}{Sections}
\crefformat{section}{#2Section~#1#3} 
\Crefformat{section}{#2Section~#1#3} 

\crefname{subsection}{\S}{\S\S}
\crefformat{subsection}{#2\S#1#3} 
\Crefformat{subsection}{#2\S#1#3} 

\theoremstyle{plain}

\newtheorem{lemma}{Lemma}[section]
\newtheorem{proposition}[lemma]{Proposition}
\newtheorem{corollary}[lemma]{Corollary}
\newtheorem{theorem}[lemma]{Theorem}
\newtheorem{conjecture}[lemma]{Conjecture}

\theoremstyle{nonumberplain}

\theoremstyle{plain}
\theorembodyfont{\upshape}
\theoremsymbol{\ensuremath{\blacklozenge}}

\newtheorem{definition}[lemma]{Definition}

\newtheorem{remark}[lemma]{Remark}
\newtheorem{convention}[lemma]{Convention}
\newtheorem{notation}[lemma]{Notation}

\crefname{definition}{definition}{definitions}
\crefformat{definition}{#2definition~#1#3} 
\Crefformat{definition}{#2Definition~#1#3} 

\crefname{notation}{notation}{notations}
\crefformat{notation}{#2notation~#1#3} 
\Crefformat{notation}{#2Notation~#1#3}

\crefname{ex}{example}{examples}
\crefformat{example}{#2example~#1#3} 
\Crefformat{example}{#2Example~#1#3} 

\crefname{remark}{remark}{remarks}
\crefformat{remark}{#2remark~#1#3} 
\Crefformat{remark}{#2Remark~#1#3} 

\crefname{convention}{convention}{conventions}
\crefformat{convention}{#2convention~#1#3} 
\Crefformat{convention}{#2Convention~#1#3}

\crefname{lemma}{lemma}{lemmas}
\crefformat{lemma}{#2lemma~#1#3} 
\Crefformat{lemma}{#2Lemma~#1#3} 

\crefname{proposition}{proposition}{propositions}
\crefformat{proposition}{#2proposition~#1#3} 
\Crefformat{proposition}{#2Proposition~#1#3} 

\crefname{corollary}{corollary}{corollaries}
\crefformat{corollary}{#2corollary~#1#3} 
\Crefformat{corollary}{#2Corollary~#1#3}

\crefname{conjecture}{conjecture}{conjectures}
\crefformat{conjecture}{#2conjecture~#1#3} 
\Crefformat{conjecture}{#2Conjecture~#1#3}

\crefname{theorem}{theorem}{theorems}
\crefformat{theorem}{#2theorem~#1#3} 
\Crefformat{theorem}{#2Theorem~#1#3} 

\crefname{assumption}{assumption}{Assumptions}
\crefformat{assumption}{#2assumption~#1#3} 
\Crefformat{assumption}{#2Assumption~#1#3} 

\crefname{equation}{}{}
\crefformat{equation}{(#2#1#3)} 
\Crefformat{equation}{(#2#1#3)}

\theoremstyle{nonumberplain}
\theoremsymbol{\ensuremath{\blacksquare}}

\newtheorem{proof}{Proof}
\newcommand\pf[1]{\newtheorem{#1}{Proof of \Cref{#1}}}

\newcommand\bQ{\mathbb Q}
\newcommand\bR{\mathbb R}
\newcommand\bS{\mathbb S}
\newcommand\bT{\mathbb T}

\newcommand\bZ{\mathbb Z}

\newcommand{\qedhere}{\mbox{}\hfill\ensuremath{\blacksquare}}

\numberwithin{equation}{section}

\renewcommand{\square}{\mathrel{\Box}}

\title{The weak Hilbert-Smith conjecture\\
  from a Borsuk-Ulam-type conjecture} \author{Alexandru Chirvasitu,
  \quad Ludwik D\k{a}browski, \quad and\quad Mariusz Tobolski}

\begin{document}

\date{}

\newcommand{\Addresses}{{
  \bigskip
  \footnotesize

  \textsc{Department of Mathematics, University at Buffalo, Buffalo, NY 14260-2900, USA}\par\nopagebreak
  \textit{E-mail address}: \texttt{achirvas@buffalo.edu}

  \medskip

  \textsc{SISSA (Scuola Internazionale Superiore di Studi Avanzati),
    Via Bonomea 265, 34136 Trieste, Italy}\par\nopagebreak
  \textit{E-mail address}: \texttt{dabrow@sissa.it}

  \medskip

  \textsc{Instytut Matematyczny, Polska Akademia Nauk,
    ul. \'Sniadeckich 8, 00-656 Warszawa, Poland}\par\nopagebreak
  \textit{E-mail address}: \texttt{mtobolski@impan.pl} 

}}

\maketitle

\begin{abstract}
We prove a number of results surrounding the Borsuk-Ulam-type conjecture of Baum, D\k{a}b\-row\-ski and Hajac (BDH, for 
short), to the effect that given a free action of a compact group $G$ on a compact space $X$, there are no $G$-equivariant maps 
$X*G\to X$ (with $*$ denoting the topological join). In particular, we prove the BDH conjecture for locally trivial 
principal $G$--bundles. The proof relies on the non-existence of $G$-equivariant 
maps $G^{*(n+1)}\to G^{*n}$, which in turn is a slight strengthening of an unpublished result of M. Bestvina and R. Edwards.
Moreover, we show that the BDH conjecture partially settles a conjecture of Ageev. In turn, the latter implies the weak 
version Hilbert-Smith 
conjecture stating that no infinite compact zero-dimensional group can act freely on a manifold such that the orbit space is finite-
dimensional.
\end{abstract}


\vspace{.5cm}


\section*{Introduction}

The fifth question on the Hilbert's famous list of problems \cite{hilbert} concerned a characterization of Lie groups. In a modern and commonly accepted form \cite[Theorem 2.7.1]{tao-hilb} the result reads as follows.

\begin{theorem}\label{h5}[Hilbert's fifth problem]
  Let $G$ be a topological group which is locally Euclidean. Then $G$ is isomorphic to a Lie group.
\end{theorem}

The above theorem was proved by Gleason \cite[Theorem 3.1]{gl-subgr} and by Montgomery and Zippin \cite[Theorem A]{mozi-subgr} in 1952. The theorem and the mathematical tools developed in the course of proving it have many applications, including the celebrated Gromov theorem on groups of polynomial growth (see \cite{tao-hilb} for a survey of applications).

It is worth noting that in Hilbert's lecture \cite{hilbert} groups were not treated abstractly, but rather as groups of transformations. Hence, one might argue that the following generalization of \Cref{h5} in fact better captures the original intent behind Hilbert's problem (see e.g. \cite[Conjecture 2.7.2]{tao-hilb}).

\begin{conjecture}\label{hs-strong}[Hilbert-Smith conjecture]
  Let $G$ be a locally compact topological group acting continuously and effectively on a connected finite-dimensional 
  topological manifold $M$. Then, $G$ is isomorphic to a Lie group.
\end{conjecture}
It is known that the Hilbert-Smith conjecture can be reduced to studying actions of the groups that are very far from being Lie, namely the groups $\bZ_p$ of $p$-adic integers for arbitrary primes $p$. This was proved by Lee \cite[Theorem 3.1]{lee-pad} using the Gleason-Yamabe theorem on the structure of locally compact groups \cite[Theorem 5']{yam-thm}, and some results of Newman \cite{new-per}. We state the reduced conjecture following \cite[Conjecture 2.7.3]{tao-hilb}.
\begin{conjecture}\label{cj.effp}[Hilbert-Smith conjecture for $p$-adic actions]
  It is not possible for a $p$-adic group $\mathbb{Z}_p$ to act continuously and effectively on a connected finite-dimensional topological manifold $M$.
\end{conjecture}

Let us recall some results that partially confirm the Hilbert-Smith conjecture. In \cite{mozi-top} Montgomery and Zippin observed that \Cref{hs-strong} is true for transitive actions on topological manifolds and for smooth actions on smooth manifolds. Repov\u{s} and \u{S}\u{c}epin \cite{resc-lip} proved it for Lipschitz actions. Then Martin announced the proof for quasiconformal actions on Riemannian manifolds in \cite{mar-conf}. There is also a proof for Hölder actions given by Maleshich \cite{mal-hol}. Recently, Pardon showed that \Cref{hs-strong} is true for three-manifolds \cite{par-3}. In the survey \cite{dran-surv} Dranishnikov gives an account of various partial results and reduces a weaker version of the conjecture to two other problems. In its full generality, the Hilbert-Smith conjecture remains unsolved.

We base our approach to the Hilbert-Smith conjecture on classifying spaces for $p$-adic groups and the theory of Menger compacta \cite{menger,best-char}. In \cite{dran-acts} Dranishnikov shows that any zero-dimensional compact metric group (e.g. $\mathbb{Z}_p$) can act freely on universal Menger compacta. Such actions can be fairly exotic, e.g. for the $n$-dimensional Menger compactum $\mu^n$ the dimension of the orbit space $\mu^n/\mathbb{Z}_p$ can exceed $n$.

Nevertheless, there is a class of actions for which $\dim(\mu^n/\mathbb{Z}_p)=n$; these are all proven isomorphic (for given $G$, in the category of $G$-spaces) by Ageev in \cite{ag-charfree}. The results of the latter paper also suggest that the actions in question satisfy a universality condition to the effect that for any free action of $G$ on an at-most-$n$-dimensional compact space $Y$, every equivariant map $Y\to \mu^n$ is approximable by an equivariant embedding. This is referred to as `strong $G$-$n$ universality' in the discussion preceding \cite[Theorem B]{ag-charfree}, and is an equivariant version of the celebrated Bestvina recognition criterion for Menger compacta (\cite{best-char} and \cite[Theorem A (b)]{ag-charfree}).

In \cite{ag-hs} Ageev states a certain conjecture about Menger compacta which if true would imply a weaker version of the Hilbert-Smith conjecture for free actions on manifolds with finite-dimensional orbit spaces. The following statement appears as \cite[Conjecture]{ag-hs}.

\begin{conjecture}\label{cj.ag}[Ageev]
  Suppose two Menger compacta $\mu^m$ and $\mu^n$ with $m>n$ are acted on freely by a non-trivial zero-dimensional compact metric group $G$. Then, there is no equivariant map $f:\mu^m\rightarrow\mu^n$.
\end{conjecture}

These ideas were followed by Yang in \cite{yang-class}, where a free Menger $\mathbb{Z}_p$-compactum is explicitly constructed.

In this paper we focus on the 
the generalized compact-Hausdorff-space Borsuk-Ulam conjecture proposed by Baum, D\k{a}browski and 
Hajac~\cite[Conjecture~2.2]{BDH}. In order to state it, recall first that for two spaces $X$ and $Y$, their {\it topological join} $X*Y$ 
is the space defined by
\begin{equation*}
  X*Y = X\times Y\times [0,1]/\sim,
\end{equation*}
where the equivalence relation identifies all $(x,y_0,0)$, $x\in X$ for fixed $y_0\in Y$, and similarly identifies all $(x_0,y,1)$, $y\in Y$ for fixed $x_0\in X$.

\begin{conjecture}\label{cj.bdh}
  Let $X$ be a compact Hausdorff space equipped with a continuous free action of a non-trivial compact Hausdorff group $G$. Then, for the diagonal action of $G$ on the topological join $X \ast G$, there does not exist a $G$-equivariant continuous map $f : X \ast G\rightarrow X$.
\end{conjecture} 

Here we show in Section~\ref{se.main} that the positive answer to \Cref{cj.bdh} would prove \Cref{cj.ag} as well. This result would then be sufficient to resolve affirmatively the finite-dimensional-orbit-space version of the Hilbert-Smith conjecture for free actions in \Cref{cor.main}.

The proposed reformulation of the problem is justified by recent work on and interest in Borsuk-Ulam-type theorems. \Cref{cj.bdh} 
is known to be true for actions of compact Hausdorff groups with torsion elements on compact Hausdorff spaces, 
which can be deduced from 
\cite[Theorem 6.2.6]{mt-bu} or \cite[Corrolary~3.1]{vol}, and was recently proved using different methods by Passer in 
\cite[Theorem~2.8]{saturated}. In \Cref{th.bu} in Section~\ref{se.bu}, we obtain a new Borsuk-Ulam-type result, 
namely for actions of a~compact Hausdorff group $G$ on a compact Hausdorff space $X$ such that $X\to X/G$ is a locally trivial 
bundle. We prove it by strengthening a theorem of M.~Bestvina and R.~Edwards (unpublished).
Note that \Cref{th.bu} includes previous Borsuk-Ulam-type results since free actions of finite groups automatically satisfy the 
locally trivial assumption. Furthermore, in \cite[Section~3]{cp} Chirvasitu and Passer propose a possible approach to 
\Cref{cj.bdh} in its full generality.


\section{Borsuk-Ulam-type theorem for locally trivial $G$-bundles}\label{se.bu}


The aim of the present Section is to outline a proof of the following Borsuk-Ulam-type result: 

\begin{theorem}\label{th.eb}
  Let $G$ be a non-trivial compact Hausdorff group, and for positive integers $n$ denote $E_nG:=G^{*(n+1)}$. Then, there are no $G$-equivariant maps $E_{n+1}G\to E_nG$. 
\end{theorem}

To the best of our knowledge, the above theorem was proved for zero-dimensional compact groups
by M. Bestvina and R. Edwards. There is an immediate corollary:
\begin{corollary}
Let $G$ be a compact Hausdorff group and for a $G$-space $X$ define
\begin{equation}
  \label{eq:3}
  {\rm ind}_G(X):={\rm min}\{n:\exists~G\text{-}{\rm map}~X\to E_nG\}.
\end{equation}
Then, ${\rm ind}_G(E_nG)=n$.
\end{corollary}

The finiteness of ${\rm ind}_G(X)$ is equivalent to the condition that $X\to X/G$ is a locally trivial principal $G$-bundle. 
Hence \Cref{th.eb} is equivalent to the following (see \Cref{prop.eq} below):
\begin{theorem}\label{th.bu}
Let $G$ be a non-trivial compact Hausdorff group and let $X$ be a compact Hausdorff $G$-space such that $X\to X/G$ is
a locally trivial principal $G$-bundle. Then there is no $G$-equivariant map $X\ast G\to X$.
\end{theorem}
\begin{remark}\label{re.tors}
The case when $G$ has non-trivial elements of finite order is known, e.g., see \cite{vol,saturated}. We also note
that \Cref{th.eb} in the case when $G$ is a finite group can be found in \cite[Theorem 6.2.5]{mt-bu}.
\end{remark}

To prove that \Cref{th.eb} and \Cref{th.bu} are equivalent, we will need the following remark regarding the 
trivializability of the principal $G$-bundle
\begin{equation*}
  E_nG\to E_nG/G=: B_nG. 
\end{equation*}
Note that the classifying space for $G$ is the direct limit of $B_nG$'s.
We slightly abuse notation and denote the coordinates on $E_nG$ by the convex combinations
\begin{equation*}
 t_0g_0+\cdots+t_ng_n,\ g_i\in G,\ t_i\in [0,1],\ \sum t_i=1,
\end{equation*}
and the action of $G$ on it is diagonal. $B_nG$ thus consists of analogous convex combinations, considered up to simultaneous translation of the $g_i$ on the right:
\begin{equation*}
t_0g_0+\cdots+t_ng_n = t_0g_0g+\cdots+t_ng_ng \text{ in } B_nG,\ \forall g_i,g\in G. 
\end{equation*}

\begin{remark}\label{le.bn}
  The bundle $E_nG\to B_nG$ can be trivialized by the cover of $B_nG$ by the $n+1$ open sets $U_i$, $0\le i\le n$, consisting respectively of the classes of combinations
  \begin{equation*}
    t_0g_0+\cdots+t_ng_n
  \end{equation*}
  for which $t_i\ne 0$. See e.g. the proof of \cite[Theorem 4.11.2]{hus}.  
\end{remark}

\begin{proposition}\label{prop.eq}
\Cref{th.eb} and \Cref{th.bu} are equivalent.
\end{proposition}
\begin{proof}
\Cref{le.bn} indicates that $E_nG$ is a locally trivial principal $G$-bundle, hence
\Cref{th.bu} implies \Cref{th.eb}.
Conversly, if $X$ is a locally trivial principal $G$-bundle, then it has finite $G$-index.
Suppose that ${\rm ind}_G(X)=n$ and that there exists a $G$-map $X\ast G\to X$. We have the following chain of $G$-maps
$$
E_{n+1}G\hookrightarrow X\ast \underbrace{G\ast\ldots\ast G}_{n+2}\to\ldots\to X\ast G\ast G\to X\ast G\to X\to E_nG,
$$
which contradicts \Cref{th.eb}.
\end{proof}

Before going into the proof of \Cref{th.eb}, we need some preparation. 
First, recall the following notion of size for a topological space. 

\begin{definition}
The Lusternik-Schnirelmann category of a topological space $X$, denoted $LS(X)$, is the minimal number $n$, for which
there exists an open covering of $X$ consisting of sets $U_i$, $i=0,1,2,\ldots,n$, such that, for every $i$, the embedding
$U_i\hookrightarrow X$ is nullhomotopic.
\end{definition}

The concept enters the present discussion via the following observation (whose routine proof we omit): 

\begin{proposition}\label{pr.lst}
The Lusternik-Schnirelmann category of the $n$-dimensional torus $\mathbb{T}^n$ is $n$. \hfill $\square$
\end{proposition}


One ingredient in the proof of \Cref{th.eb} will be the existence of locally trivial principal $G$-bundles over $\mathbb{T}^n$ that are non-trivial over non-trivial loops. We will thus need to prove the existence of such bundles. By a {\it non-trivial loop} on a topological space $X$ we mean a non-nullhomotopic continuous map $\alpha:[0,1]\to X$ such that $\alpha(0)=\alpha(1)$.

All of our topological spaces will be Hausdorff and locally path connected, to ensure a well-behaved theory of covering spaces (see e.g. \cite[$\S$II.2]{mssy}, where the notion is referred to as local {\it arcwise} or {\it pathwise} connectivity). Let $X$ be a connected space, $G$ a compact group, and $\Gamma$ a discrete group acting on $X$ (from the left) freely and {\it properly discontinuously}, i.e. such that every point has a neighborhood $U$ with
\begin{equation}\label{eq:gu}
  \gamma U\cap U=\emptyset,\ \forall \gamma\in \Gamma, \gamma\ne 1,
\end{equation}
(\cite[p. 136]{mssy}).

We also fix a morphism $\phi:\Gamma\to G$, and regard $\Gamma$ as acting on $G$ on the left by translations via~$\phi$. Note that $\Gamma$ acts properly discontinuously on $X\times G$ as well: for any $U\subset X$ satisfying (\ref{eq:gu}), so does $U\times G$. Moreover, the right action of $G$ on the right-most component of $X\times G$ commutes with the left $\Gamma$-action and hence descends to an action on
\begin{equation*}
  E:= \Gamma \backslash (X\times G). 
\end{equation*}
It is not difficult to check that we have

\begin{lemma}\label{le.loc-triv}
In the setting above, $E$ is a locally trivial principal $G$-bundle over $Y:=\Gamma \backslash X$.   
\end{lemma}
\begin{proof}
  Let $y\in Y$, choose a point $x\in \pi^{-1}(y)$ where
  \begin{equation*}
    \pi:X\to \Gamma\backslash X = Y
  \end{equation*}
  is the projection, and let $x\in U\subset X$ be an open subset satisfying (\ref{eq:gu}). Then, $V=\pi(U)$ is an open neighborhood of $y$ and condition (\ref{eq:gu}) ensures that the restriction of $E\to Y$ to $E|_{V}\to V$ is isomorphic to $V\times G\to V$.
\end{proof}

Now cover $Y$ with open sets $V_i$ of the form $\pi(U)$ for open $U\subset X$ satisfying (\ref{eq:gu}). $\pi:X\to Y$ is a covering map and the opens $V_i$ are chosen so that the preimage $\pi^{-1}(V_i)$ is the disjoint union of open subsets $U_{i\ell}\subset X$ mapped onto $V_i$ homeomorphically by $\pi$. Identify each restriction $E|_{V_i}$ with the trivial principal $G$-bundle $U_{i\ell}\times G$ via $\pi$ for one of the $U_{i\ell}$'s:
\begin{equation*}
  h_i : U_{i\ell}\times G\cong V_i\times G\to E|_{V_i}. 
\end{equation*}
This is an atlas for the bundle in the sense of \cite[Chapter 5, Definition 2.1]{hus}, and the choice of the atlas makes it clear that the resulting transition functions $g_{ij}:V_i\cap V_j\to G$ defined by
\begin{equation*}
  (h_i^{-1}\circ h_j)(y,g) = (y, g_{ij}g)
\end{equation*}
actually factor through $\phi:\Gamma\to G$. In other words, the structure group of the fiber bundle $E\to Y$ can be reduced from $G$ to the discrete group $\Gamma$ via $\phi$. The bundle thus admits a flat connection, and the corresponding monodromy morphism
\begin{equation*}
  \pi_1(Y)\to \Gamma\to G 
\end{equation*}
is nothing but the composition of $\phi:\Gamma\to G$ with the canonical morphism $\pi_1(Y)\to \Gamma$ resulting from the quotient map
\begin{equation*}
  \pi:X\to \Gamma\backslash X\cong Y. 
\end{equation*}

Before continuing, we need more preparation. Following \cite[Chapter 8, Part 2]{hm} we denote by $G_a$ the normal subgroup of the compact group $G$ consisting of elements that are path-connected to the identity. Note that in general $G_a$ is not closed, so we will disregard the quotient topology on
\begin{equation*}
  \pi_0(G):=G/G_a,
\end{equation*}
considering it simply as an abstract group.

\begin{definition}\label{def.incmprs}
  A principal $G$-bundle over $Y$ is {\it incompressible} if its restriction to every non-nullhomotopic loop is non-trivial. 
\end{definition}

\begin{lemma}\label{le.sc}
  In the setting above, suppose $X$ is simply connected and $\phi:\Gamma\to G\to \pi_0(G)$ is a group monomorphism.
  Then, the bundle $E\to Y$ constructed above is incompressible. 
\end{lemma}
\begin{proof}
  By \cite[Theorem~1.1~(ii)]{kambton-flat-bun}, the equivalence class under inner automorphisms of $G$ of the homotopy class of the monodromy map is an invariant of bundles. In particular, for trivial bundles the map is homotopic to the trivial morphism.
  
On the other hand, under the hypotheses of the lemma the monodromy
  \begin{equation*}
    m:\pi_1(Y)\cong \Gamma\to G
  \end{equation*}
  cannot be homotoped to $1\in G$ on any non-trivial element $s$ of $\pi_1(Y)$ because $m(s)$ belongs to $G/ G_a$. In conclusion, the above paragraph implies that the restriction of $E$ to the circle via a map
  \begin{equation*}
    \bS^1\to Y
  \end{equation*}
  sending the generator of $\pi_1(\bS^1)$ to $s\ne 1$ in $\pi_1(Y)$ is non-trivial. 
\end{proof}

\begin{proposition}\label{le.emb}
  If $G$ is a non-trivial torsion-free compact group then, for every positive integer~$n$, $G$ contains a copy of $\mathbb{Z}^n$ which embeds into $\pi_0(G)$ via
  \begin{equation*}
    \bZ^n\to G\to \pi_0(G). 
  \end{equation*}
\end{proposition}
\begin{proof}
  We do this in stages.

  {\bf Step 1: $G$ itself contains $\bZ^n$.} The torsion freeness of $G$ ensures that we have an embedding of the additive group $\mathbb{Z}_p$ of $p$-adic integers into $G$; it thus suffices to argue that $\mathbb{Z}^n$ embeds into $\mathbb{Z}_p$ for every $n$. This, however, is clear: the quotient field
  \begin{equation*}
    \mathbb{Q}_p = \mathbb{Z}_p\left[\frac 1p\right]
  \end{equation*}
  of $\mathbb{Z}_p$ is a field of characteristic zero and continuum cardinality, and hence is a vector space of dimension $2^{\aleph_0}$ over $\mathbb{Q}$. If the $p$-adic rationals $v_1$ up to $v_n$ are linearly independent over $\mathbb{Q}$, then for sufficiently large $k$ the elements $p^kv_i$ belong to $\mathbb{Z}_p$ and generate a group isomorphic to $\mathbb{Z}^n$.

  {\bf Step 2: $G=\widehat{\bQ}$.} Since $G$ is abelian and $\bZ^n$ is projective in the category of abelian groups, any embedding $\bZ^n\to \pi_0(G)$ lifts to an embedding $\bZ^n\to G$ as in the statement. In conclusion, it suffices to argue that $\pi_0(G)$ contains a copy of $\bZ^n$.
  
  By \cite[Theorem~8.30~(iii)]{hm}, we have the following isomorphism
$$
\pi_0(\widehat{\bQ})\cong {\rm Ext}(\bQ,\bZ).
$$
We know that ${\rm Ext}(\bQ,\bZ)\cong \bR$ from \cite{wie-ext} and, as $\bR$ is a $\bQ$-vector space of dimension \
$2^{\aleph_0}$, using a similar argument as in Step 1, the result follows.
  
{\bf Step 3: the full result.} As seen in Step 1, we are done if $G$ is totally disconnected. If it is not, then the connected component $G_0$ of the identity is a non-trivial connected torsion-free compact group. Restricting our attention to it, we can assume $G$ itself is connected to begin with.

By \cite[Theorem~9.24 (ii)]{hm}, being compact and connected, $G$ is isomorphic to the quotient group $(Z_0(G)\times \Pi_j S_j)/D$, 
where $Z_0(G)$ is the connected component of the center of $G$, all the $S_j$'s are simple simply connected Lie groups
and $D$ is a totally disconnected central subgroup of the domain of the surjection 
\begin{equation}\label{eq:zsg}
  Z_0(G)\times \prod S_j\to G.
\end{equation}
Moreover, \Cref{eq:zsg} identifies the $Z_0(G)$ factor with its copy inside $G$. 

Since torsion elements are dense in compact connected Lie groups and $G$ is assumed torsion-free, the factors $S_j$ all map to 
$G$ trivially. It follows that in fact \Cref{eq:zsg} factors through $Z_0(G)\to G$. In other words, $G$ must be abelian, and we can 
work with its discrete Pontryagin dual $\widehat{G}$. 

The connectedness of $G$ now implies that $\widehat{G}$ is torsion-free, while the fact that $G$ itself is torsion-free means that $
\widehat{G}$ is divisible. The only divisible torsion-free abelian groups are the vector spaces over $\bQ$, i.e. direct sums $
\bQ^{\oplus S}$ for some set $S$. This in turn implies that
\begin{equation*}
  G\cong \widehat{\bQ}^S.
\end{equation*}
In particular $G$ is abelian, and the same argument as in Step 2 reduces the problem to showing that $\bZ^n$ embeds in $
\pi_0(G)$. This, however, follows from Step 2, since $\pi_0(\widehat{\bQ})$ embeds into $\pi_0(G)$ via one of the factor 
embeddings $\widehat{\bQ}\subset G$. 
\end{proof}

\begin{corollary}\label{cor.torus}
  If $G\ne \{1\}$ is torsion-free and $n$ is a positive integer, then there are incompressible $G$-bundles on the $n$-dimensional torus $\mathbb{T}^n$. 
\end{corollary}
\begin{proof}
This follows from  \Cref{le.sc} applied to $\Gamma=\bZ^n\cong \pi_1(\bT^n)$.
\end{proof}


We are now ready for the

\begin{proof}[of \Cref{th.eb}]
  As noted in \Cref{re.tors}, the case when $G$ has torsion is settled in a stronger form. For this reason, throughout the proof we assume that $G$ is torsion-free; then, \Cref{le.emb} and \Cref{cor.torus} apply.
  
  Let $E\to \mathbb{T}^n$ be a bundle as in \Cref{cor.torus}. The space $E_nG$ has vanishing homotopy groups $\pi_i$, $1\le i<n$, and hence the bundle $E_nG\to B_nG$ is {\it $n$-universal} in the sense of \cite[$\S$19.2]{stnrd}: for any locally trivial 
principal $G$-bundle $X\to Y$ with ${\rm dim}(Y)\leq n$ there exists a map of $G$-bundles
\begin{equation*}
\label{univ}
\xymatrix{ 
X \ar[r] \ar[d]
& 
E_n G  \ar[d]\\
Y \ar[r] &B_n G
}
\end{equation*}
Now suppose there exists a $G$-map $E_nG \to E_{n-1}G$ and consider the following commutative diagram. 
\begin{equation*}
\label{diag}
\xymatrix{ 
E \ar[r] \ar[d]
& 
E_n G  \ar[d] \ar[r]
& E_{n-1} G  \ar[d]\\
\mathbb{T}^n \ar[r] &B_n G \ar[r] & B_{n-1} G
}
\end{equation*}

Recall that by construction, the bundle $E\to \mathbb{T}^n$ is nontrivial over every nontrivial loop. We know from Lemma \ref{le.bn} that $B_{n-1}G$ can be covered with $n$ open trivializing sets. Pulling those sets back to $\mathbb{T}^n$, we obtain a open trivializing cover consisting of sets $U_i$, $i=0,1,2,\ldots,n-1$. No $U_i$ can contain a nontrivial loop, and hence the maps $U_{i}\to \mathbb{T}^n$ are nullhomotopic: 

Indeed, the non-existence of non-trivial loops in the connected components $U_{i\ell}$ of $U_{i}$ ensures that the image of $\pi_1(U_{i\ell})\to \pi_1(\mathbb{T}^n)$ is trivial. This then implies that the map $U_{i\ell}\to \mathbb{T}^n$ factors through the contractible universal cover $\mathbb{R}^n\to \mathbb{T}^n$. The conclusion follows. 

We now have a contradiction: trivializing $E\to \mathbb{T}^n$ with $n$ open sets $U_i$ whose embeddings into $\mathbb{T}^n$ are nullhomotopic contradicts \Cref{pr.lst}. 
\end{proof}
\begin{remark}
We would like to remark that the core idea of the proof of \Cref{th.eb} is due to M.~Bestvina and R.~Edwards,
while \Cref{le.sc} and \Cref{le.emb} contain new results that enabled us to generalize the theorem to arbitrary compact Hausdorff 
groups.
\end{remark}

\section{Relation to the Hilbert-Smith conjecture}\label{se.main}

\subsection{Preliminaries}\label{se.prel}

From now on, all of our topological spaces are at the very least second countable and metrizable (or equivalently, separable and metrizable). We record this formally:

\begin{convention}\label{cv.cv}
  Henceforth, unless specified otherwise, we only consider separable metrizable topological spaces. 
\end{convention}

{\it Compacta} are compact Hausdorff metrizable topological spaces. We denote by $\mu^n$ the universal $n$-dimensional Menger compactum characterized abstractly in a number of ways in \cite{best-char}.

Since we have to work extensively with quotients of spaces by group actions, it will be necessary for said quotients to be metrizable (and separable) as well. Since all of our actions are proper and all of our topological groups are locally compact and separable (because they act freely on locally compact separable metric spaces) \cite[Theorem B and Corollary 1]{ant-ney} confirm that indeed $X/G$ is metrizable and separable; we take this for granted repeatedly below.

We use \cite{hw,eng} as our sources of background for the dimension theory of separable metrizable spaces. We will not recall any of the precise definitions of dimension here (henceforth denoted by $\dim$), and refer instead e.g. to \cite[Definition III 1]{hw} or \cite[Definitions 1.1.1, 1.6.1, 1.6.7]{eng} for the various notions of dimension (small or large inductive as well as covering dimension) and \cite[Theorem 1.7.7]{eng} for verification that they all coincide for separable metric spaces.

Suffice it to say the concept captures the usual intuition, and specializes to the standard notion of dimension for manifolds.  The universality of $\mu^n$ referred to above consists in the fact that it is $n$-dimensional, and contains a homeomorphic copy of every $n$-dimensional compactum.

Following \cite[Definition 1]{ag-hs}, we recall

\begin{definition}\label{de.univ}
  Let $n$ be a non-negative integer and $G$ a topological group acting freely on a space $X$. Then, $X$ equipped with the action in question is {\it free $n$-universal} (or simply {\it universal} for short if $n$ is understood) if for any free $G$-space $Y$ with
  \begin{equation*}
    \dim(Y/G)\le n
  \end{equation*}
  and any closed $G$-invariant subspace $Z\subseteq Y$ every equivariant map $Z\to X$ admits an extension to an equivariant map $Y\to X$.
\end{definition}

In \Cref{se.main} below we will make crucial use of the following result (see \cite[Theorem 1]{ag-hs}).

\begin{theorem}\label{th.mug-univ}
  A Menger compactum $\mu^n$ equipped with a free action by a zero-dimensional compact group is free $n$-universal in the sense of \Cref{de.univ}. \qedhere
\end{theorem}

Finally, we will on occasion refer to {\it ANR} spaces (for absolute neighborhood retract). These are spaces $X$ with the property that whenever embedded homeomorphically $\iota:X\to Y$ as a closed subspace of a topological space $Y$, there is an open neighborhood $U$ of $X$ in $Y$ that retracts onto $X$ in the sense that we have a diagram of the form
\begin{equation*}
  \begin{tikzpicture}[auto,baseline=(current  bounding  box.center)]
    \path[anchor=base] (0,0) node (x) {$X$} +(4,0) node (y) {$Y$} +(2,-.5) node (u) {$U$};
         \draw[right hook->] (x) to[bend left=16] node[pos=.5,auto] {$\scriptstyle \iota$} (y);
         \draw[->] (x) to [bend left=10] (u);
         \draw[right hook->] (u) to [bend right=6] (y);
         \draw[->] (u) to [bend left=10] node[pos=.5,auto] {$\scriptstyle r$} (x);
  \end{tikzpicture}
\end{equation*}
for some $r:U\to X$ so that the $X$-based loop is the identity.  Metrizable manifolds are automatically ANR, e.g. by \cite[Theorem V.7.1]{hu} or the main result of \cite{caut-anr}.

\subsection{The weak Hilbert-Smith conjecture}
In \cite[Conjecture 1.2]{dran-surv} 
a version of the Hilbert-Smith conjecture was presented in its weak form, whereby free actions are substituted for effective ones, 
and an additional constraint is imposed on the orbit space. That is, one assumes that, for a free action of a
compact group $G$ on a finite-dimensional topological manifold $M$, the orbit space $M/G$ is finite-dimensional. At first sight
this is counterintuitive since it disagrees with the well-known equation
$$
\dim(M/G)=\dim(M)-\dim(G),
$$
which holds for example for free actions of compact Lie groups.
However, in the case when $G=\mathbb{Z}_p$, the situation changes drastically. Although $p$-adic integers
are zero-dimensional, Smith showed that the dimension of $M/\mathbb{Z}_p$ is not equal to $\dim(M)$~\cite{sm-per}.
Next, combining the result of C.T.~Yang~\cite{y-p} and Alexandroff's theorem about the coincidence of
the cohomological and covering dimension  \cite{al-dim}, one can deduce that for a free action of $\mathbb{Z}_p$
on an $n$-manifold we have that either $\dim(M/\mathbb{Z}_p)=n+2$ or it is infinite. Therefore,
proving the Hilbert-Smith conjecture for finite-dimensional orbit spaces would exclude the first possibility.

Along the lines of \cite{tao-hilb}, we somewhat strengthen the statement of \cite[Conjecture 1.2]{dran-surv} by allowing our 
groups to be {\it locally} compact, so long as the action is well-behaved enough for the resulting orbit spaces to be locally compact 
Hausdorff spaces. Namely, we assume the actions to be {\em proper}. Let us list four equivalent definitions of when an action of a 
locally compact group $G$ on a locally compact Hausdorff space $X$ is proper:
\begin{itemize}
\item the map $G\times X\to X\times X,\; (g,x)\mapsto (xg,x)$ is a proper map,  i.e. preimages of compact sets are compact,
\item for any pair of points $x, y\in X$ there exist open neighborhoods $U$ of $x$ and $V$ of $y$ such that the set $$U.V:= \{g\in G\,|\, Ug\cap V\neq\emptyset \}$$ is relatively compact, i.e. its closure is compact,
\item if $K$ is a compact subset of $X$, then the set $K.K$ is compact.
\item if $K$ and $K'$ are compact subsets of $X$ then $K.K'$ is compact.
\end{itemize}
The equivalence of the above definitions follows from \cite[Proposition~7 p.~104 and Proposition~7 p.~255]{bourbaki-proper} and \cite[Theorem~1.2.9]{palais}.

\begin{remark}\label{re.prop}
Note that properness, like freeness, is a hereditary property, i.e. it descends to open or closed subgroups.
Indeed, let $G$ be a locally compact Hausdorff topological group acting properly on a locally compact Hausdorff space $X$.
This means that for any compact subset $K$ of $X$, the set $K.K$ is compact in $G$.
Recall that an open subgroup of a topological group is automatically closed, so we only
need to prove the claim for closed subgroups. Suppose that $H$ is a closed subgroup of $G$
and take any compact subset $K$ in $X$. We want to prove that the set
$$
\{h\in H:Kh\cap K\neq\emptyset\}=K.K\cap H
$$
is compact. This is clear as the above set is a closed subset of $K.K$, 
which is compact by properness of the action of $G$ on $X$.
\end{remark}

We now state the weaker Hilbert-Smith conjecture:
\begin{conjecture}\label{cj.hs}
  Let $G$ be a locally compact group acting freely and properly on a connected finite-dimensional 
  topological manifold $M$ such that the orbit space $M/G$ is finite dimensional. Then $G$ is a Lie group.
\end{conjecture}

An examination of the argument presented in \cite[$\S$2.7.2]{tao-hilb} will show that the reduction of the general Hilbert-Smith 
conjecture to its $p$-adic version (\Cref{cj.effp}) also goes through in the present setting of free actions and finite-dimensional orbit 
spaces:

\begin{conjecture}\label{cj.hsp}
  Let $p$ be a prime number and $M$ a connected finite-dimensional topological manifold. Then $\bZ_p$ cannot act freely on $M$ so that $M/\bZ_p$ is finite-dimensional.
\end{conjecture}

\subsection{Reduction to compact groups}
We set about reducing \Cref{cj.hs} to \Cref{cj.hsp}  gradually, starting by transporting the discussion over to compact (as opposed to locally compact) groups:

\begin{theorem}\label{le.cpct}
  If \Cref{cj.hs} holds for compact groups then it does in general. 
\end{theorem}

We need some preparation, in the form of a series of auxiliary results. First, given that the various reductions involve passing from actions of $G$ to the restricted actions of various subgroups of $G$ and the resulting conjectures stipulate the finite-dimensionality of orbit spaces, we will make repeated use of the following remark.

\begin{lemma}\label{le.osg}
  If \Cref{cj.hs} holds for an open subgroup $H\le G$ then it does for $G$ as well. 
\end{lemma}
\begin{proof}
Under hypothesis of \Cref{cj.hs}, $G$ acts freely and properly on a topological manifold $M$, hence, by \Cref{re.prop},
$H$ is again acting freely and properly.
If $H$ is a Lie group, then so is $G$, since $H$ is its open subgroup.
The only observation needed here is that if $M/G$ is finite-dimensional then so is $M/H$. In fact, we will prove more:

  {\bf Claim: Under the hypotheses, $M/H\to M/G$ is a local homeomorphism.} Given this, it follows from the fact that the property of being of dimension $\le n$ is a local property (\cite[Definition 1.1.1]{eng}) that the two spaces in question have equal dimension; it thus remains to prove the claim. 

  We denote the various canonical projection maps as in the diagram
  \begin{equation*}
      \begin{tikzpicture}[auto,baseline=(current  bounding  box.center)]
    \path[anchor=base] (0,0) node (m) {$M$} +(4,0) node (mg) {$M/G$.} +(2,-.5) node (mh) {$M/H$};
         \draw[->] (m) to[bend left=16] node[pos=.5,auto] {$\scriptstyle \pi_G$} (mg);
         \draw[->] (m) to[bend right=6] node[pos=.5,auto,swap] {$\scriptstyle \pi_H$} (mh);
         \draw[->] (mh) to[bend right=6] node[pos=.5,auto,swap] {$\scriptstyle \pi_{H,G}$} (mg);
  \end{tikzpicture}
  \end{equation*}

Let $p_H\in M/H$ be an arbitrary point. We have to show that some open neighborhood $p_H\in U\subseteq M/H$ is mapped homeomorphically onto its image by $\pi_{H,G}$. 

We propose to do this as follows:
  \begin{enumerate}
    \renewcommand{\labelenumi}{(\alph{enumi})}
  \item fix some $p\in \pi_H^{-1}(p_H)$;
  \item show that there is some open neighborhood
    \begin{equation*}
      p\in V\subseteq M 
    \end{equation*}
    such that the set
    \begin{equation*} 
    V.V=\{g\in G\ |\ Vg\cap V\ne\emptyset\}
    \end{equation*}
    is contained in $H$;
  \item deduce from this containment that $VgH$ and $VH$ are disjoint if $gH\ne H$ are different cosets, and finally
  \item take $U\subseteq M/H$ to be $\pi_H(V)$.     
  \end{enumerate}

Once (b) is in place (c) follows immediately: if $VgH\cap VH\ne \emptyset$ then
\begin{equation*}
  vgh = v'
\end{equation*}
for some choice of $v,v'\in V$, etc. We then have $gh\in V.V\subseteq H$ by (b), and hence $g\in H$. 

Next, we argue that the choice of $U$ proposed in (d) meets the requirements (i.e. is mapped homeomorphically onto $\pi_{H,G}(U)=\pi_G(V)$). To see this, note that
\begin{equation}\label{eq:2}
  \pi_{H,G}^{-1}(\pi_{H,G}(U)) = \pi_H(VG) = \pi_H\left(VH\sqcup \bigcup_g VgH\right)
\end{equation}
where
\begin{itemize}
\item the union $\bigcup_g$ is over representatives of the cosets $gH\ne H$ and
\item the union $\sqcup$ is disjoint by (c).  
\end{itemize}
On the other hand, because the parenthetic sets on the right hand side of \Cref{eq:2} are invariant under the action of $H$, that right hand side is
\begin{equation*}
  \pi_H(VH)\sqcup \pi_H\left(\bigcup_g VgH\right) = U\sqcup \pi_H\left(\bigcup_g VgH\right).
\end{equation*}
In other words, $\pi_{H,G}^{-1}(\pi_{H,G}(U))$ breaks up as a disjoint union between $U$ and another open set. This suffices to conclude that
\begin{equation*}
  \pi_{H,G}|_U:U\to \pi_{H,G}(U)
\end{equation*}
is an isomorphism.

It thus remains to find a neighborhood $V\ni p$ as in (b). In order to do this, note first that the properness of the action ensures the existence of a compact neighborhood $M\supseteq W\ni p$ such that $W.W\subset G$ is compact. Now consider the net $(V)$ of compact neighborhoods
\begin{equation*}
p\in V\subseteq W  
\end{equation*}
ordered by reverse inclusion (i.e. $V_1\succeq V_2$ if $V_1\subseteq V_2$). Suppose that $V.V\not\subseteq H$ for any such $V$. This means that each of the sets $V.V\cap (G - H)$ is non-empty. Since these sets are contained in the compact set $W.W$, we can thus find a subnet $(V_{\alpha})_{\alpha}$ of $(V)$ and elements
\begin{equation*}
  g_{\alpha}\in {V_{\alpha}.V_{\alpha}}\cap (G-H) 
\end{equation*}
converging to some $g\in G-H$. But then, because $(V)$ converges to $p\in M$ and hence so does $(V_{\alpha})$, we have $pg=p$. Given that $g\in G-H$, this contradicts the freeness of the action and finishes the proof.
\end{proof}

To reduce the problem to compact groups, we need another lemma.
\begin{lemma}\label{le.dim-bd}
  Let $H$ be a locally compact group acting freely and properly on a locally compact separable metrizable space $X$ such that both $X$ and $X/H$ are finite-dimensional. We then have
  \begin{equation*}
    \dim(X/N)<\infty
  \end{equation*}
  for every compact normal subgroup $N\le H$.
\end{lemma}
\begin{proof}
  Since $N$ is normal, there is an action of $H$ on $X/N$ given by $[x]_{X/N}\cdot h:=[xh]_{X/N}$. We begin with

  {\bf Claim: The action of $H$ on $X/N$ is proper.} To see this denote by $\pi$ the canonical surjection $X\to X/N$ and let
  \begin{equation*}
    K,K'\subseteq X/N
  \end{equation*}
  be two compact subsets. We then have
  \begin{equation*}
    K.K' = (\pi^{-1}(K).\pi^{-1}(K'))N,
  \end{equation*}
  and the right hand side is compact because
  \begin{itemize}
  \item $N$ is compact;
  \item this in turn implies that $\pi$ is proper and hence the preimages $\pi^{-1}(-)$ on the right hand side are compact;
  \item finally, the properness of the original $H$-action on $X$ shows that $\pi^{-1}(K).\pi^{-1}(K')$ is compact.
  \end{itemize}
This finishes the proof of the claim.   

The orbits of the action of $H$ on $X/N$ are homeomorphic to $H/N$ and the orbit space is homeomorphic to $X/H$. Given the claim, \cite[Theorem~3.10]{ant-dim-proper} (which requires properness) is applicable and we have
  \begin{equation}\label{eq:1}
    \dim(X/N)\le \dim(X/H)+\dim(H/N). 
  \end{equation}
Since $H$ acts freely on a metrizable and finite-dimensional space, $H$ is automatically metrizable and finite-dimensional. By \cite[Corollary 2]{mst-x} $H\to H/N$ is a locally trivial bundle and therefore $\dim(H/N)<\infty$. It follows that both summands on the right hand side of \Cref{eq:1} are finite, finishing the proof. 
\end{proof}

We are now ready for the

\pf{le.cpct}
\begin{le.cpct}
Assume that \Cref{cj.hs} is true for compact groups.
  We follow the strategy pursued in \cite[$\S$2.7.2]{tao-hilb}, using the Gleason-Yamabe theorem. Recall that the latter states that 
  every locally compact group $G$ has an open subgroup $H\le G$ which in turn has a normal compact subgroup $N$ such that 
  $H/N$ is Lie.

  \Cref{le.osg} has already reduced the problem to its analogue for $H$.
  Now, if $N$ is Lie, then so is $H$, as it is an extension of a Lie group $H/N$ by Lie group $N$.
  Again, we only need to prove the implication
  \begin{equation}\label{eq:gzp}
    \dim(M/H)<\infty\Rightarrow \dim(M/N)<\infty. 
  \end{equation}
  This, however, is nothing but \Cref{le.dim-bd} ($N$ is a compact normal subgroup of $H$).
\end{le.cpct}

\subsection{Reduction to p-adic integers}
Finally, we can reduce \Cref{cj.hs} to the $p$-adic case. 

\begin{theorem}\label{le.red}
  \Cref{cj.hsp} implies \Cref{cj.hs}. 
\end{theorem}

We once more need some preparatory remarks. In the same spirit as \Cref{le.dim-bd} we have 

\begin{lemma}\label{le.bd-dim-0}
  Let $X$ be a separable metrizable space acted upon freely by the compact group $G$. Then, for every closed subgroup $H\le G$ we have
  \begin{equation*}
    \dim(X/H)\le \dim(X/G) + \dim(G/H). 
  \end{equation*}
\end{lemma}
\begin{proof}
  If either $X/G$ or $G/H$ is infinite-dimensional then there is nothing to prove, so we assume both are finite-dimensional.

  Because $G$ is compact,
  \begin{equation*}
    \pi_{H,G}:X/H\to X/G
  \end{equation*}
  is a closed, finite-dimensional map (in the sense that its fibers, homeomorphic to $G/H$, are finite-dimensional). Coupled with the observation made in the discussion following \Cref{cv.cv} that all of our quotient spaces are metrizable, this ensures that \cite[Theorem 4.39]{eng} applies and delivers the conclusion.
\end{proof}

\begin{remark}\label{re.ant}
  The $H=\{1\}$ case of \Cref{le.bd-dim-0} is a particular instance of \cite[Theorem 1.3]{ant-dim-proper}.  
\end{remark}

Finally:

\pf{le.red}
\begin{le.red}
First, observe that we can reduce \Cref{cj.hs} to compact groups by \Cref{le.cpct}.
We prove the contrapositive of \Cref{le.red}, namely
we assume that there is a free action of a compact non-Lie group $G$ on a connected finite-dimensional manifold $M$
with $M/G$ finite-dimensional, and we find an action of $\bZ_p$ on $M$ as in~\Cref{cj.hsp}. By \cite[Theorem~3.1]{lee-pad}, we know that $G$ contains $\mathbb{Z}_p$ 
and by \Cref{le.bd-dim-0} applied to $H=\bZ_p$ we have
  \begin{equation*}
    \dim(M/G)<\infty \Rightarrow \dim(M/\bZ_p)<\infty.
  \end{equation*}
\end{le.red}

\subsection{A Borsuk-Ulam-type conjecture implies the weak Hilbert-Smith conjecture}
Next, as per \Cref{le.red}, our aim will be to show that \Cref{cj.bdh} implies \Cref{cj.hsp}. 
For the convenience of the reader, we include below a diagrammatic depiction of the logical dependence of the results. 
We denote conjectures, lemmas and theorems by $\mathrm{C}$, $\mathrm{L}$ and $\mathrm{T}$ 
respectively, and corollaries by $\mathrm{Cor}$.

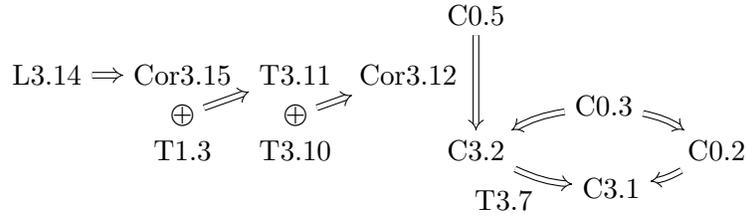
\begin{equation*}
     \begin{tikzpicture}[auto,baseline=(current  bounding  box.center)]
       \path[anchor=base]
       (0,0)            node (17)   {$\mathrm{L}3.14$}
       +(1.8,0)         node (18)       {$\mathrm{Cor}3.15$}
       +(1.8,-.5)       node (++)               {$\oplus$}
       +(1.8,-1)        node (7)   {$\mathrm{T}1.3$}
       +(3.3,-1)                node (14)   {$\mathrm{T}3.10$}
       +(3.3,-.5)               node (+)                {$\oplus$}
       +(3.3,0)                 node (15)   {$\mathrm{T}3.11$}
       +(4.8,0)                 node (16)   {$\mathrm{Cor}3.12$}
       +(5.7,.8)                node (5)    {$\mathrm{C}0.5$}
       +(5.7,-1)                node (12)       {$\mathrm{C}3.2$}
       +(7.4,-.4)               node (3)                {$\mathrm{C}0.3$}
       +(7.5,-1.5)      node (11)   {$\mathrm{C}3.1$}
       +(8.9,-1)                node (2)    {$\mathrm{C}0.2$};
       \draw[-implies, double equal sign distance] (17) -- (18);
       \draw[-implies, double equal sign distance] (++) -- 
       (15);
       \draw[-implies, double equal sign distance] (+) -- (16);
       \draw[-implies, double equal sign distance] (5) -- (12);
       \draw[-implies, double equal sign distance] (3) to[bend right=10] (12);
       \draw[-implies, double equal sign distance] (3) to[bend left=10] (2);
       \draw[-implies, double equal sign distance] (12) to[bend  
right=10] node[pos=.5,auto,swap] {$\mathrm{T}3.7$} (11);
       \draw[-implies, double equal sign distance] (2) to[bend left=10] (11);
   \end{tikzpicture}
\end{equation*}

\vspace{2mm}
We follow the plan laid out in \cite{ag-hs}, as we now recall. First, we paraphrase \cite[Theorem 4]{ag-hs} in a form that will suit us here. This slightly more precise form of the statement is easily extracted from the proof in loc.\,cit.

\begin{theorem}\label{th.4}
  Let $p$ be a prime number and let the $p$-adic group $\bZ_p$ act freely on the Menger compacta $\mu^m$ and $\mu^n$ for some $m>n$ so that
  \begin{equation*}
    \dim(\mu^m/\bZ_p)<\infty. 
  \end{equation*}
  Suppose furthermore that there are no $\bZ_p$-equivariant maps $\mu^m\to \mu^n$.

  Then, for any ANR space $M$ acted upon freely by $\bZ_p$ we have $\dim(M/\bZ_p)>n$. \qedhere
\end{theorem}

Given \Cref{th.4}, \cite{ag-hs} proposes to attack \Cref{cj.hsp} via \Cref{cj.ag}. However, somewhat weaker version of the latter is sufficient for our purposes:

\begin{theorem}\label{th.main}
  Let $m>n$ be non-negative integers, and let the zero-dimensional non-trivial compact group $G$ act freely on $\mu^m$ and $\mu^n$ so that $\dim(\mu^n/G)\le m$. Suppose that there are no $G$-equivariant maps $\mu^n*G\to \mu^n$
for diagonal action on the join.

  Then, there are no $G$-equivariant maps $\mu^m\to \mu^n$.
\end{theorem}

Notice that \Cref{th.main} assumes that \Cref{cj.bdh} is satisfied for a zero-dimensional compact metric group acting on a compact Hausdorff space of finite covering dimension. Before going into the proof, let us record the resolution of the weak Hilbert-Smith conjecture.

\begin{corollary}\label{cor.main}
  \Cref{cj.bdh} implies \Cref{cj.hsp}.
\end{corollary}
\begin{proof}
  Let $m>n$ be arbitrary positive integers.

  We first apply \Cref{th.main} to actions of $\bZ_p$ on $\mu^m$ and $\mu^n$ constructed as in \cite{dran-acts}, having the properties
\begin{equation*}
  \dim(\mu^m/\bZ_p)=m,\ \dim(\mu^n/\bZ_p)=n. 
\end{equation*}
The non-existence of an equivariant map $\mu^m\to \mu^n$ resulting from this theorem then ensures that the hypotheses of \Cref{th.4} are satisfied for this specific choice of free actions, and hence $\dim(M/\bZ_p)>n$. Since $n$ was arbitrary, this concludes the proof that the dimension of the orbit space cannot be finite.
\end{proof}

\begin{notation}\label{not.diag}
  In the context of a group $G$ acting on spaces $X$ and $Y$, we write $X\times_\Delta Y$ (with the group and actions being understood) to denote the Cartesian product $X\times Y$ equipped with the diagonal $G$-action.
\end{notation}

\begin{lemma}\label{le.orb-dim}
  Let $X$ be a compact space equipped with a free action by a compact group $G$. Then, we have the inequality
  \begin{equation}\label{eq:dim-ineq}
    \dim(X*G/G)\le \max(\dim(X/G),\dim(X)+1). 
  \end{equation}
\end{lemma}
\begin{proof}
  The orbit space $X*G/G$ on the left hand side of \Cref{eq:dim-ineq} is the union of the two closed subspaces $X/G$ (at the endpoint $1\in [0,1]$; see the definition of joins in the introduction) and $\{*\}\cong G/G$ at $0\in [0,1]$, and the open subspace
  \begin{equation*}
    X\times_\Delta G\times J/G\cong (X\times_\Delta G/G)\times J
  \end{equation*}
  where $J=(0,1)$ and the $\Delta$ subscript is explained in
  \Cref{not.diag}.

\cite[Corollary 1 to Theorem III 2]{hw} implies the estimate
\begin{equation*}
  \dim(X*G/G)\le \max(\dim(X/G),\dim((X\times_\Delta G/G)\times J)). 
\end{equation*}
Hence, for our purposes, it suffices to prove that we have
\begin{equation}\label{eq:dim-ineq-aux}
  \dim((X\times_\Delta G/G)\times J) \le \dim(X)+1. 
\end{equation}
First, notice that the subadditivity of dimension under taking Cartesian products (\cite[Theorem III 4]{hw}) together with $\dim(J)=1$ bounds the left hand side of \Cref{eq:dim-ineq-aux} above by
\begin{equation*}
  \dim(X\times_\Delta G/G)+1. 
\end{equation*}
Finally, the desired conclusion follows from the fact that 
\begin{equation*}
  (x,g)\mapsto (xg,g) 
\end{equation*}
implements an isomorphism of $G$-spaces between $X\times G$ with the right hand factor action and $X\times_\Delta G$. This, in turn, ensures
\begin{equation*}
  X\times_\Delta G/G\cong X,
\end{equation*}
and hence 
\begin{equation*}
  \dim(X\times_\Delta G/G)=\dim(X). 
\end{equation*}
This finishes the proof of the lemma.
\end{proof}

As an immediate consequence of \Cref{le.orb-dim}, we obtain

\begin{corollary}\label{cor.orb-dim}
  Let $m>n$ be non-negative integers. For a free action of a compact group $G$ on $\mu^n$ with $\dim(\mu^n/G)\le m$ we have $\dim(\mu^n*G/G)\le m$.  \qedhere
\end{corollary}

\pf{th.main}
\begin{th.main}
  The freeness of all actions in sight imply that all orbits are equivariantly homeomorphic to $G$ itself.

  Now, our assumption that $\dim(\mu^n/G)\le m$ shows via \Cref{cor.orb-dim} below that we have
\begin{equation}\label{eq:ineq-works}
    \dim(\mu^n*G/G)\le m.
\end{equation}

The universality of the $G$-action on $\mu^m$ provided by \Cref{th.mug-univ} and \Cref{eq:ineq-works} ensure that any equivariant homeomorphism from a $G$-orbit in $\mu^n*G$ onto a $G$-orbit in $\mu^m$ extends to an equivariant map
  \begin{equation*}
    \mu^n*G\to \mu^m. 
  \end{equation*}
  Composition with an equivariant map $\mu^m\to \mu^n$ as in the statement of the theorem would then imply the existence of an equivariant map $\mu^n*G\to \mu^n$, contradicting the assumption.
\end{th.main}

\begin{remark}
  \Cref{th.4,th.main} are stated in such generality to emphasize their connection with Ageev's original formulation \cite{ag-hs}. However, working with the universal actions on Menger compacta $\mu^n$ such that $\mathrm{dim}(\mu^n/\mathbb{Z}_p)=n$ would still confirm \Cref{cj.hsp}.

  As mentioned in the introduction, the actions in question are proved unique up to isomorphism (for given $n$) in \cite[Theorem B]{ag-charfree} and their existence is proven by \cite[Theorem 1]{dran-acts}. Remark 1 on p.~228 therein argues that the orbit space for the action on $\mu^n$ being constructed in the proof of the theorem is indeed $n$.
\end{remark}


\subsection*{Acknowledgements}

We acknowledge a suggestion by A.~Volovikov that the conjecture made in \cite{BDH} could prove \Cref{cj.ag}, and discussions with H.~Toru\'nczyk and with S.~Spie\.{z}, who suggested using the universality of Menger compacta. We thank M.~Bestvina and R.~Edwards for confirmation of their unpublished result (Theorem~\ref{th.eb}) and sharing their inside regarding its proof.  We would also like to thank S.~M.~Ageev for discussions about the paper \cite{ag-hs}.  Finally, we are grateful to P.~M.~Hajac and B.~Passer for numerous conversations on the contents of this paper and related topics contributing considerably to its improvement.

This work is part of the project Quantum Dynamics sponsored by EU-grant RISE 691246 and Polish Government grants  
3542/H2020/2016/ and  328941/PnH/201.

A.C. was partially supported by NSF grants DMS-1565226 and DMS-1801011.

L.D. is grateful for his support at IMPAN provided by Simons-Foundation grant 346300 and a Polish Government MNiSW 2015–2019 matching fund.

Finally, M.T. was partially supported by the project Diamentowy Grant No.~DI2015 006945 financed by the Polish Ministry of Science and Higher Education.


\def\polhk#1{\setbox0=\hbox{#1}{\ooalign{\hidewidth
  \lower1.5ex\hbox{`}\hidewidth\crcr\unhbox0}}}


\Addresses

\end{document}